\documentclass[11pt]{amsart}
\usepackage{amsmath,amsthm,epsfig,amssymb,hyperref}
\usepackage{amsthm,amsfonts,amssymb,amscd}

\usepackage[all]{xy}
\newtheorem{thm}[subsection]{Theorem}
\newtheorem{prop}[subsection]{Proposition}

\newtheorem{lem}[subsection]{Lemma}
\newtheorem{corol}[subsection]{Corollary}
\newtheorem{rem}[subsection]{Remark}
\theoremstyle{definition}
\newtheorem{Def}[subsection]{Definition}

\newtheorem{exam}[subsection]{Example}
\newtheorem{proposition-definition}[subsection]{Proposition-Definition}

\newcommand{\OOO}{{\mathcal O}}

\renewcommand\square{\frame{\phantom{{\large x}}}}

\usepackage[backend=biber, style=numeric, sorting=nyt]{biblatex}
\title{Generalized determinantal representation of hypersurfaces}
\author{A. El Mazouni}
\address{Laboratoire de Math\'ematiques de Lens EA 2462
Facult\'e des Sciences Jean Perrin
Rue Jean Souvraz, SP18
F-62307 LENS  Cedex France}
\email{abdelghani.elmazouni@univ-artois.fr}

\author{D. S. Nagaraj}
\address{Indain Institute of Science Education and Research, 
Srinivasapuram, Yerpedu Mandal
Tirupati Dist, Andhra Pradesh, India – 517619.}
\email{dsn@labs.iisertirupati.ac.in}

\author{Supravat Sarkar}
\address{Fine Hall, Princeton University, Princeton, NJ 700108, USA.}
\email{ss6663@princeton.edu}
\begin{document}

\begin{abstract}  

In this article we extend the notion of determinantal representation of hypersurfaces 
to the determinantal representation of sections of the determinant line
bundle of a vector bundle. We give several examples, and prove some necessary 
conditions for existence of determinantal representation. As an application, 
we show that for any integer $d \geq 1,$ there is an 
indecomposable vector bundle $E_d$ of rank $2$ on $\mathbb{P}^2$ such that 
almost all curves of degree $d$ of $\mathbb{P}^2$ arise as the degeneracy
loci of a pair of holomorphic sections of  $E_d$, upto an automorphism of 
$\mathbb{P}^2$. We use this result to obtain a
linear algebraic application.
\end{abstract}

\maketitle

\section{Introduction} \setcounter{page}{1} 
Throughout we work over the field $\mathbb{C}$ of complex numbers. 
For a vector space $W$ of dimension $n$  over $\mathbb{C}$ 
and an integer $0< r< n$ we denote by $Gr(r, W)$ the Grassmannian  
variety of $r$-dimensional subspaces of $W.$ 

Determinantal representation of homogeneous polynomials have been
studied for quite some time in literature, see for example 
\cite{beauville2000determinantal},\cite{branden2011obstructions},\cite{LuJu}, 
\cite{chiantini2014determinantal}, \cite{netzer2012polynomials}, 
\cite{quarez2012symmetric} and \cite{vinnikov1989complete}.
We want to extend the notion of determinantal representation of homogeneous 
polynomials to the determinantal representation of sections of the determinant line
bundle of a vector bundle. For a smooth projective variety $X$ with an ample line 
bundle $H$ and a vector bundle $E$ on $X$ of rank $d \geq 2$, call $E$ to be $H$-
abundant if for all $m>>0$, and $D$ a general member of the complete linear system
given by the line bundle $\textrm{det } E(mH)$, there is an automorphism $\phi$ of 
$X$ such that 
$\phi^* D$ is the degeneracy loci of $d$ sections of $E(mH)$. Call $E$ to be 
(semi-)abundant, if $E$ is $H$-abundant for (some) all ample $H$. Of course, if $E$ 
and $H$ are homogeneous vector bundles, the automorphism $\phi$ is not needed.

The main result of \cite{LuJu} shows that split vector bundles on $\mathbb{P}^1$ and 
$\mathbb{P}^2$ are abundant. One of the main goals of this paper is to give several 
examples of indecomposable abundant vector bundles on $\mathbb{P}^2$.
\begin{thm}\label{thm1}
Let $N$ be the rank two vector bundle on $\mathbb{P}^2$ obtained by taking the 
quotient of ${\mathcal O}_{\mathbb{P}^2}^2 \oplus {\mathcal O}_{\mathbb{P}^2}(1) $
 by the subbundle $(x,y,z^2){\mathcal O}_{\mathbb{P}^2}(-1).$ Then $N$ is abundant.
\end{thm}
\begin{thm}\label{thm2}
For $k\geq 1$, let $M_k$ be the syzygy bundle of $\mathcal{O}_{\mathbb{P}^2}(k)$,
that is, the dual $M_k^*$ is the kernel of the evaluation map 
$H^0(\mathbb{P}^2,\mathcal{O}_{\mathbb{P}^2}(k))\otimes_k \mathcal{O}_{\mathbb{P}^2}
\to \mathcal{O}_{\mathbb{P}^2}(k)$. Then $M_k$ is abundant for all $k.$
\end{thm}

Note that $M_k$ is indecomposable, it is in fact a stable bundle.
\begin{corol}\label{tgt}
    The tangent and cotangent bundle of $\mathbb{P}^2$ are abundant.
\end{corol}

\begin{corol}\label{cor1}
For any $d\geq 1$ there is an indecomposable vector bundle $E_d$ of rank
$2$ on $\mathbb{P}^2$ such that  almost all curves of degree $d$ in 
$\mathbb{P}^2$ arise as the degeneracy loci of a pair of holomorphic 
sections of  $E_d,$ up to an automorphism of $\mathbb{P}^2$.
\end{corol}
Next we show the bundles considered in \cite{ray2020projective} are abundant.
\begin{thm}\label{thm3}
For integer $2\leq r\leq 4$, let $E_r$ be the dual of the kernel of the surjection \begin{equation*}
\mathcal{O}_{\mathbb{P}^2}^{r+2}\to\mathcal{O}_{\mathbb{P}^2}(1)^{2}     
\end{equation*}
given by 
\begin{equation*}
 \underline{\alpha}\longmapsto (x_{0}\alpha_{0}+x_{1}\alpha_{1}+x_{2}\alpha_{2}, x_{0}\alpha_{r-1}+x_{1}\alpha_{r}+x_{2}\alpha_{r+1}),   
\end{equation*} where $x_0,x_1,x_2$ is the standard basis of 
$H^0(\mathcal{O}_{\mathbb{P}^2}(1)).$ Then $E_r$ is abundant.
\end{thm}

Now we show that the class of varieties possessing a (semi-)abundant vector bundle is very restrictive.
\begin{thm}\label{thm4}
Suppose there is an semi-abundant vector bundle on $X$ and $\textrm{dim} X=n$. Then 
$n\leq 2$. If $n= 2$, then either 
$\kappa(X)=-\infty$, or $\kappa(X)=0$ and $X$ is minimal. If
$n= 2$ and there is an abundant vector bundle on $X$, then 
$-K_X$ is rationally effective, that is, 
$H^0(X,\mathcal{O}(-mK_X))\neq 0$ for some $m>0.$
\end{thm}

Finally, we give more examples of varieties possessing (semi-)abundant bundles.

\begin{thm}\label{thm5}
\noindent\begin{enumerate}
    \item The trivial vector bundle of rank $2$ on $\mathbb{P}^1\times \mathbb{P}^1$ is semi-abundant.
    \item If $C$ is a smooth projective curve which can be embedded in either $\mathbb{P}^2$ or $\mathbb{P}^1\times \mathbb{P}^1$. Then $\mathcal{O}_C^2$ is semi-abundant.
    \item If $C$ is a smooth projective curve in $\mathbb{P}^2$, then for all nonnegative integers $k$, $M_k|_{C}$ is semi-abundant. Here $M_k$ is defined in Theorem \ref{thm2}.
\end{enumerate}
\end{thm}
 
\section{Proof of The Theorem \ref{thm1}}

For $n\geq 0$, we have an exact
sequence 
 \begin{equation}\label{euler}
 0 \to {\mathcal O}_{\mathbb{P}^2}(n-1) \to {\mathcal O}_{\mathbb{P}^2}(n)^2\oplus{\mathcal O}_{\mathbb{P}^2}(n+1) \to N(n) \to 0,
 \end{equation} 
of vector bundles on $\mathbb{P}^2,$ where the first map is given 
by $f\mapsto (f.x,f.y,f.z^2).$
Since $H^1(\mathcal{O}_{\mathbb{P}^2}(n-1))=0,$
we conclude that the map 
$$H^0({\mathcal O}_{\mathbb{P}^2}(n)^2\oplus{\mathcal O}_{\mathbb{P}^2}(n+1)) \to H^0(N(n)) $$
is surjective. Thus every section of  $N(n)$ comes from a section of $\mathcal {O}_{\mathbb{P}^2}(n)^2\oplus{\mathcal O}_{\mathbb{P}^2}(n+1)$
 \begin{Def}\label{generic} For a vector bundle $F$ on $\mathbb{P}^2$ a two dimensional subspace $V$ of $H^0(\mathbb{P}^2, F(n))$ is called generically point-wise
 linearly independent  (GPLI, for short), if there is a point $p \in {\mathbb{P}^2}$ and 
 $v_1, v_2 \in V$ such that the tangent vectors $v_1(p) $ and $ v_2 (p) $ are independent.
 \end{Def}
 
 \begin{exam} The sequence \eqref{euler} gives an exact sequence of vector spaces :
 $$ 0 \to H^0({\mathcal O}_{\mathbb{P}^2}(n) ) \to H^0({\mathcal O}_{\mathbb{P}^2} (n+1)^3) \to H^0(T_{\mathbb{P}^2}(n) )\to 0.$$
 The image of sections $(X^{n+1},0,0)$ and $(Y^{n+1},0,0)$  in  $ H^0(T_{\mathbb{P}^2}(n+1) )$ are linearly independent 
 but not generically point-wise linearly independent.
 \end{exam} 
 
 \begin{rem}
 a)  If $v_1$ and $v_2$ are two linearly independent section of $N(n)$ such that the subspace $V = <v_1,v_2>$ 
 of $H^0(\mathbb{P}^2, N(n))$  is {\em not} GPLI
then the  vector space  $\wedge^2(V)$ maps to zero in $H^0(\mathbb{P}^2, \OOO_{\mathbb{P}^2}(2n+2))$

b) Let 
$$U = \{  [V] \in Gr(2, H^0(N(n))) \,\, | \,\, V  \, \text{is a GPLI} \} .$$ 
Then $U$ is an open  subset of  $Gr(2, H^0(N(n))) :$ 

In fact, if $G = Gr(2, H^0(N(n))),$ then the complement of $U$ is the base locus of the linear system determined by the subspace $H^0(\mathbb{P}^2, \mathcal{O}_{\mathbb{P}^2}(2n+2))^*$ 
in  $H^0(G, \mathcal{O}_G(1))\simeq \wedge^2(H^0(N(n)))^*.$
\end{rem}

Thus $[V] \mapsto [\wedge^2(V)]$ defines a rational map 
\[ \Psi_n' : Gr(2, H^0(N(n))) \cdots\to \mathbb{P}(H^0(\mathbb{P}^2, \OOO_{\mathbb{P}^2}(2n+2))). \]
which is defined exactly on the open set $U.$ This map can be described as follows:
If $[V] \in U$ and $(f_1,f_2,f_3)$ and $(g_1,g_2,g_3)$ are two sections of the 
vector bundle ${\mathcal O}_{\mathbb{P}^2}(n)^2\oplus{\mathcal O}_{\mathbb{P}^2}(n+1) $ that maps to a basis of the vector
space $V$ then using sequence (\ref{euler}) we can identify the curve defined by 
$\wedge^2(V)$ with the zero locus of the curve defined by the determinant of the matrix
\[ \begin{pmatrix}
g_1 & g_2 & g_3\\
f_1 & f_2 & f_3\\
x & y & z^2
\end{pmatrix}
\]
Hence it suffices to show that for a general curve $C$ of degree $2n+2$ in $\mathbb{P}^2$, there is an automorphism $\phi$ of $\mathbb{P}^2$ such that $\phi^*(C)$ is the zero locus of the 
 determinant of a matrix of the form 

 \[ \begin{pmatrix}
g_1 & g_2 & g_3\\
f_1 & f_2 & f_3\\
x & y & z^2
\end{pmatrix}
\] where $g_1, g_2, f_1,f_2$ are all homogeneous polynomials  of degree $n$, $g_3$ and $f_3$ are homogeneous of degree $n+1$. Now {\cite[Main Theorem]{LuJu}}
 says that a general plane curve $C$ of degree $2n+2$ is the zero locus of the 
 determinant of a matrix of the form 
 \[ \begin{pmatrix}
g_1 & g_2 & g_3\\
f_1 & f_2 & f_3\\
l & m & Q
\end{pmatrix}
\]
where $g_i$ and $f_i$ are as above, $l$ and $m$ are homogeneous linear polynomials, 
and $Q$ is a homogeneous quadratic polynomial. As $C$ is general we can assume that 
$l$ and $m$ are linearly independent, and $Q$ is not in the ideal $(l,m)$ of 
$k[x,y,z]$. So, replacing $C$ by $\phi^*(C)$ for an automorphism $\phi$ of 
$\mathbb{P}^2$, we may assume that $l=x$ and $m=y.$ Now by adding a polynomial linear 
combination of first and second column to the third column (which does not change the
determinant), we may assume $Q=az^2$ for some $a\in k$. As $Q\not\in (l,m)$, we have
$a\neq 0$. So, dividing the third column by $a$, we get the result. This completes the 
proof of Theorem (\ref{thm1}).
$\hfill{\square}$
\begin{rem}
    Composing $\Psi_n'$ by the action of $\textrm{Aut }\mathbb{P}(H^0(\mathbb{P}^2, \OOO_{\mathbb{P}^2}(2n+2))^*)$ yields a rational map

\[
 \Psi_n : Gr(2, H^0(N(n)))\times \textrm{Aut}(\mathbb{P}^2)
  \cdots\to \mathbb{P}(H^0(\mathbb{P}^2, \OOO_{\mathbb{P}^2}(2n+2))^*).  
 \] We have shown in the proof that in fact $\Psi_n$ is dominant for all $n\geq 0$, not just for only sufficiently large $n.$ 
\end{rem} 
\begin{rem} Note that the bundle $N$ is a restriction to a hyperplane  of the Null correlation bundle on $\mathbb{P}^3,$ constructed and studied by W. Barth \cite{Ba}. The Null correlation bundle is the rank-two bundle on $\mathbb{P}^3$ with the property of being stable but only semi-stable when restricted to any linear subspace, which is unique up to automorphisms of $\mathbb{P}^3$ and up to tensoring by line bundles.
\end{rem}
 \section{Proof of The Theorem \ref{thm2}}

We have an exact
sequence 
 \begin{equation*}
 0 \to {\mathcal O}_{\mathbb{P}^2}(-k) \to H^0(\mathbb{P}^2,\mathcal{O}_{\mathbb{P}^2}(k))^*\otimes_k \mathcal{O}_{\mathbb{P}^2} \to M_k \to 0.
 \end{equation*}
 Twisting by $\mathcal {O}_{\mathbb{P}^2}(n)$ gives the exact sequence

 \begin{equation}\label{syzygy}
 0 \to {\mathcal O}_{\mathbb{P}^2}(n-k) \to H^0(\mathbb{P}^2,\mathcal{O}_{\mathbb{P}^2}(k))^*\otimes_k \mathcal{O}_{\mathbb{P}^2}(n) \to M_k (n)\to 0.
 \end{equation}

Let $m=h^0(\mathbb{P}^2,\mathcal{O}_{\mathbb{P}^2}(k))=\binom{k+2}{2},$ and 
$\{P_1,P_2,...,P_m\}$ a basis of $H^0(\mathbb{P}^2,\mathcal{O}_{\mathbb{P}^2}(k)).$
Hence we can identify \ref{syzygy} as a short exact sequence 

\begin{equation}
 0 \to {\mathcal O}_{\mathbb{P}^2}(n-k) \to  \mathcal{O}_{\mathbb{P}^2}(n)^m \to M_k (n)\to 0,
 \end{equation}

where the first map is given by $f\mapsto (fP_1,fP_2,\cdots fP_m)$.

Now we finish the proof in the same way as in Theorem \ref{thm1}. By \cite{LuJu}, for each $n\geq 0$ a general curve $C$ of degree $(m-1)n+k$ is the zero locus of the determinant of an $m\times m$ matrix $A$ of homogeneous polynomials such that each entry in the last row has degree $k$, and all other entries has degree $n$. Since $C$ is general,we can assume that the polynomials in the last row of $A$ are linearly independent, hence form 
a basis of $H^0(\mathbb{P}^2,\mathcal{O}_{\mathbb{P}^2}(k))$. So, post-multiplying $A$ by an element of $GL_m(\mathbb{C})$, (which changes the determinant only up to a nonzero scalar multiple), we can assume the last row of $A$ is $(P_1,P_2,...,P_m).$ This completes the proof in the same way as in the proof of Theorem \ref{thm1}
 $\hfill{\square}$

 \begin{rem}
    For $k\geq 1, n\geq 0$, we have a rational map

\[
 T_{k,n} : Gr(m-1, H^0(M_k(n)))
  \cdots\to \mathbb{P}(H^0(\mathbb{P}^2, \OOO_{\mathbb{P}^2}((m-1)n+k))^*).  
 \] We have shown in the proof that in fact $T_{k,n}$ is dominant for all $n\geq 0$, 
 not just for only sufficiently large $n.$ Also, the automorphism as in the 
 definition of abundant vector bundle is not needed here. This last fact is anyway 
 clear, as $M_k$ is homogeneous vector bundle.
\end{rem}

\textbf{Proof of Corollary \ref{tgt}:}

Note that the tangent and cotangent bundles are line bundle twists of $M_1$.
$\hfill{\square}$

\textbf{Proof of Corollary \ref{cor1}:}

Note that $det (T_{\mathbb{P}^2}(n)) $ is equal to $ \OOO_{\mathbb{P}^2}(2n+3)$ and 
$det(N(n))$ is equal to $ \OOO_{\mathbb{P}^2}(2n+2).$ For any integer $d \geq 1,$ set
$E_d = N(k),$  if $d=2k$ and $E_d= T_{\mathbb{P}^2}(k-1)$ if $d=2k+1.$ Then the result
follows at once 
from Corollary (\ref{tgt}) and the proof of Theorem (\ref{thm1}). 
$\hfill{\square}$

In the remaining of this section, we discuss some consequence and examples related to
the abundance of $T_{\mathbb{P}^2}.$ For $n\geq -1$, let 
$$\Phi_n=T_{1,n+1}: Gr(2, H^0(T_{\mathbb{P}^2}(n))) \cdots\to \mathbb{P}
(H^0(\mathbb{P}^2, \OOO_{\mathbb{P}^2}(2n+3))).$$ 
So,$\Phi_n$ is dominant for all $n\geq -1.$ 

 Note that if $W$ is space of all plane  curves of degree $2n+3,$ then  for 
 $C \in W,$  the tangent space 
 $$T_{C}(W)=H^0(\mathcal O_{C}(2n+3)).$$ 
  (More generally if $D$ is divisor on a variety $X$  then $\mathcal{O}_D(D)$ 
  is the normal bundle of $D$ in $X$ and the Zariski tangent space to the scheme of 
  all divisors linearly equivalent to $D$ on $X$ can identified with 
  $H^0(X,\mathcal{O}_D(D))$).
   
Let $V$ be a two dimensional GPLI  subspace of 
 $H^0(T_{\mathbb{P}^2}(n)).$  
The Zariski tangent space of the Grassmannian 
$Gr(2, H^0(T_{\mathbb{P}^2}(n))$ at $V$ can be identified with 
$\text{Hom}_{\mathbb{C}}(V , H^0(T_{\mathbb{P}^2}(n))/V)$
and tangent space to\\ 
$\mathbb{P}(H^0(\mathbb{P}^2, \mathcal{O}_{\mathbb{P}^2}(2n+3)))$ at a
curve $C$ of degree $2n+3$ can be  identified with 
$H^0(C, \mathcal{O}_{\mathbb{P}^2}(2n+3)|_C)$
 \begin{prop}\label{tangent} The tangent map 
 $$(d\Phi_n)_{[V]} :\text{Hom}_{\mathbb{C}}(V , H^0(T_{\mathbb{P}^2})/V)) \to H^0(C, 
 \mathcal{O}_{\mathbb{P}^2}(2n+3)|_C) $$
 is given by 
 $$\varphi \mapsto (v_1 \wedge \varphi(v_2) - v_2 \wedge \varphi(v_1))|_C,$$
  where $v_1, v_2 \in V$ is any basis of $V.$
 \end{prop}
 
  {\bf Proof:} 
  Let $V =\langle v_{1},v_{2}\rangle  \subset  H^0(T_{\mathbb{P}^2}) $  be a GPLI.  
  An infinitesimal variation of the subspace 
  $V = \langle v_1, v_2\rangle $ is of the form 
  $$V_{\varphi}= \langle v_1+t\varphi(v_{1}) , v_2+t\varphi(v_{2})\rangle$$
   for some $\varphi \in Hom(V,H^0(T_{\mathbb{P}^2})/V),$  
  and $(v_{1}+t\varphi(v_{1}))\wedge(v_{2}+t\varphi(v_{2}))=v_{1}\wedge
  v_{2}+t(v_{1}\wedge \varphi(v_{2})-v_{2}\wedge \varphi(v_{1}))+\cdots.$ 
 
  If we choose another basis of $V$ like,
  $\tilde{v_{1}}=\alpha_{11}v_{1}+\alpha_{21}v_{2}$, 
  $\tilde{v_{2}}=\alpha_{12}v_{1}+\alpha_{22}v_{2}$, 
  $ \tilde{v_{1}}\wedge \varphi(\tilde{v_{2}})-\tilde{v_{2}}\wedge \varphi(\tilde{v_{1}})= detM(v_{1}\wedge \varphi(v_{2})-v_{2}\wedge \varphi(v_{1}))$ where 
  $M=\left ( \begin{array}{cc}\alpha_{11}&\alpha_{12}\\ \alpha_{21}&\alpha_{22}\\ \end{array}\right).$ The defining equation of $C$ is also changed by $det(M).$
  Thus we see that 
 $$(d\Phi_n)_{[V]}(\varphi)= (v_1 \wedge \varphi(v_2) - v_2 \wedge \varphi(v_1))|_C.$$  $\hfill{\square}$
 \begin{corol}\label{cor2}
  If $[V] \in Gr(2, H^0(T_{\mathbb{P}^2}(n))$ is general then the tangent map 
 $$(d\Phi_n)_{[V]} :\text{Hom}_{\mathbb{C}}(V , H^0(T_{\mathbb{P}^2})/V)) \to H^0(C, \mathcal{O}_{\mathbb{P}^2}(2n+3)|_C).$$  
 \end{corol}
{\bf Proof:} Over the field of complex numbers, morphism of non-singular varieties is generically smooth (See {\cite[Chapter III, Corollary 10.7]{Ha}}) and hence  submersion when it is dominant. This proves the corollary. $\hfill{\square}$

\begin{rem}\label{rem1}
 If $f : X \cdots\to Y $ is rational map of non-singular projective varieties and at a point $p \in X$
  the tangent map $Tf_p : T_p \to T_{f(p)}$ is surjective, then it follows
  that the morphism $f$ is dominant (See {\cite[Chapter III, \S 10]{Ha}}).
\end{rem}
Using the Remark \ref{rem1} we can get an alternating  proof of abundance of $T_{\mathbb{P}^2}$ if we can produce one two dimensional GPLI subspace $V$ of
$H^0(T_{{\mathbb{P}}^2}(n))$   for which the linear map 
$$ 
(d\Phi_n)_V: Hom(V,H^0(T_{\mathbb{P}^2}(n)/V)) \to H^0(\mathcal O_{C}(2n+3)) $$ 
defined by 
$$ \varphi \mapsto (v_{1}\wedge \varphi(v_{2})-v_{2}\wedge \varphi(v_{1})|_V$$ is surjective.

For $n=0$ we have the following example:

{\bf Example 1):}
Take $v_{1},v_{2} \in H^0(T_{\mathbb{P}^2})$ where $v_{1}=\psi(x,2y,3z)$ $v_{2}=\psi(y,z,x)$, where  
$$\psi: H^0(\mathcal O_{\mathbb{P}^2}(1)^3) \to H^0( T_{\mathbb{P}^2} )$$
is the linear map obtained from the Euler exact sequence. If $V$
is the two dimensional subspace of $H^0( T_{\mathbb{P}^2} )$ generated by $v_{1}$ and $v_{2},$
then the equation of the curve $C=\Phi_0(V)$ is
\[
F=  \left |   \begin{array}{ccc}x&2y&3z\\y&z&x\\x&y&z\\ \end{array} \right |  = x^2y-2xz^2+y^2z. 
\] 
It is easy to see that   $C=Z(F)$ is a non-singular cubic of $\mathbb{P}^2$ and 
$$H^0(\mathcal O_{C}(3))=\langle x^3,y^3,z^3,xyz,x^2z,y^2z,xy^2,xz^2,yz^2\rangle .$$ 
We claim that 
the linear map 
\[ (d\Phi_0)_V: Hom(V,H^0(T_{\mathbb{P}^2})/V) \to H^0(\mathcal O_{C}(3)) 
\]
is surjective:
For, if $\varphi \in Hom(V,H^0(T_{\mathbb{P}^2})/V) $ then we have 
$\varphi(\psi(v_{1}))=\psi(f_{1},f_{2},f_{3})$ modulo $V$ and 
$\varphi(\psi(v_{2}))=\psi(g_{1},g_{2},g_{3})$ modulo $V,$ for some 
$f_i,g_i \in H^0(\mathcal O_{\mathbb{P}^2}(1)), \,\, (1\leq i \leq 3).$ 
 If 
 $(g_1,g_2,g_3) = (0,0,0)$ and $(f_{1},f_{2},f_{3})=(0,x,0)$ 
 (resp. $(0,0,y)$, $(z,0,0)$) we obtain $\varphi_{1}$ 
 (resp. $ \varphi_{2}\,$,$\,\varphi_{3}$) such that  
 $(d\Phi_0)_V(\varphi_{1})= xyz-x^3,$
(resp. $(d\Phi_0)_V(\varphi_{2})= -y^3+xyz$, 
$(d\Phi_0)_V(\varphi_{1})= -z^3 +xyz$) up to a non-zero scalar. 
If $\varphi_{4}$ (resp. $\varphi_{5}$, $\varphi_{6}$,
$\varphi_{7}$, $\varphi_{8}$, $\varphi_{9}$)is given by 
$(f_1,f_2,f_3)=(0,0,0)$ and 
$(g_{1},g_{2},g_{3})=(y,0,0)$ (resp.$(z,0,0),\,(0,x,0),\,\\ (0,y,0),\,(0,z,0)$,
$\,(0,0,y) $) then 
$(d\Phi_0)_V(\varphi_{4})$ $=y^2z$ 
(resp. $(d\Phi_0)_V(\varphi_{5})=z^2y$,
$(d\Phi_0)_V(\varphi_{6})= x^2z$, $(d\Phi_0)_V(\varphi_{7}) =xyz$,
 $(d\Phi_0)_V(\varphi_{8})= z^2x$, $(d\Phi_0)_V(\varphi_{9})=y^2x$) up to
a non-zero scalar.
$\hfill{\square}$

{\bf Example 2):}
Let
\[ 0 \to O_{\mathbb{P}^2}(-1) \stackrel{(x,y,z^2)}{\longrightarrow} 
O_{\mathbb{P}^2}^2 \oplus O_{\mathbb{P}^2}(1) \stackrel{\psi}{\to} 
N \to 0 \]
be the exact sequence of vector bundle on ${\mathbb{P}^2}.$
Take $u_{1},u_{2} \in H^0(N)$ where $u_{1}=\psi(0, 1, y)$ 
$u_{2}=\psi(1,0,x),$ 
Let $C$ be the zero locus of the section $u_1\wedge u_2$ of the line bundle $det(N) = {\mathcal O}_{\mathbb{P}^2}(2).$
If  
\[ 
F=  \left |  \begin{array}{ccc}x& y& z^2\\ 0& 1 & y\\ 1& 0 & x\\ \end{array} \right |  = x^2 +y^2-z^2,
\] 
then $C=V(F)$ and it is a smooth conic in $\mathbb{P}^2.$
$$H^0(\mathcal O_{C}(2))=\langle x^2,y^2,xy,xz,yz\rangle ,$$ 
As in Example 1 by evaluating at different choices of elements of 
$Hom(V, H^0(N)/V)$ we can show that the linear map
\[
(d\Psi_0)_V : Hom(V, H^0(N)/V) \to H^0(C,\mathcal{O}_C(2))
\]
is surjective. $\hfill{\square}$

\section{An algebraic application}
Let $\mathsf{m}$ be the maximal ideal of $\mathbb{C}[x,y,z]$ generated 
by $x,y,z.$ If $S$ a finite set of homogeneous polynomials of degree $d$
such that the ideal $I$ generated by $S$ is $\mathsf{m}$-primary (i.e.,
$\sqrt{I}= \mathsf{m}$), then there is an integer $k_0$ such that for
all integer $k \geq k_0$ the ideal $\mathsf{m}^k$ is contained in $I.$
It is interesting to find an upper bound for smallest $k_0$ in terms of 
$d.$ In general, such bounds are obtained by effective nullstellensatz,
see for example \cite{kollar1988sharp}. For some special choices of 
homogeneous polynomials, one can get better bounds. Theorem 
(\ref{thm33}) below will provide such an upper bound for some special
sets of homogeneous polynomials.

Let $n \geq 0$ be an integer. 
Let $f_1, f_2, f_3$ and $g_1, g_2, g_3$ be homogeneous polynomials of
degree $n+1$ in three variables $x,y,z$ over the field of complex 
numbers. Assume that the subsets 
of $\mathbb{P}^2$ for which 
\[Z(f_1y-f_2x, f_1z-f_3x,f_2z-f_3y) \]
and
\[ Z(g_1y-g_2x, g_1z-g_3x,g_2z-g_3y) \]
 are disjoint.  Set $ U$ to be the subspace of 
$H^0(\mathcal{O}_{\mathbb{P}^2}(n+2))$ generated by the homogeneous polynomials
\[f_1y-f_2x, f_1z-f_3x,f_2z-f_3y, g_1y-g_2x, g_1z-g_3x,g_2z-g_3y.\] 

\begin{thm}\label{thm33}
 The linear map 
  \[ F_n : H^0(\mathcal{O}_{\mathbb{P}^2}(n+1)) \otimes U \to H^0(\mathcal{O}_{\mathbb{P}^2}(2n+3))\]
  induced by multiplication $ (f , h) \mapsto f.h$ is surjective for a generic choice of 
  $f_1, f_2, f_3$ and $g_1, g_2, g_3.$
\end{thm}

{\bf Proof:} If $f_1, f_2, f_3$ and $g_1, g_2, g_3$ are general then 
images of 
$(f_1, f_2, f_3)$ and $(g_1, g_2, g_3)$ in $T_{\mathbb{P}^2}(n)$
under the map 
\[{\mathcal O}_{\mathbb{P}^2}(n+1)^3 \to T_{\mathbb{P}^2}(n) \to 0\]
in (\ref{euler}) gives a two dimensional subspace $V$ of
$H^0(T_{\mathbb{P}^2}(n))$
for which the the tangent map
\[
(d\Phi_n)_V: Hom(V,H^0(T_{\mathbb{P}^2}(n)/V)) \to H^0(\mathcal O_{C}
(2n+3))  \]
is surjective (see, Remark (\ref{rem1})). If $W$ is the two dimensional
subspace of
$$H^0({\mathcal O}_{\mathbb{P}^2}(n+1)^3)$$
generated by $v_1=(f_1, f_2, f_3)$ and $v_2=(g_1, g_2, g_3)$ then we see
that there a commutative diagram
\[\begin{array}{ccc}
Hom(W,H^0(\mathcal{O}_{\mathbb{P}^2}(n+1))^3))& \stackrel{F_n}
{\longrightarrow} & H^0(\mathcal O_{\mathbb{P}^2}(2n+3))\\
\downarrow & {} & \downarrow \\
Hom(V,H^0(T_{\mathbb{P}^2}(n)/V))& \stackrel{(d\Phi_n)_V}
{\longrightarrow} &H^0(\mathcal O_{C}(2n+3))
\end{array}\]
of linear maps, where $F_n(\phi) = v_1 \wedge \phi(v_2) - v_2 \wedge 
\phi(v_1),$
first vertical arrow is obtained from identifying $W$ and $V$ and second
vertical arrow
is the natural restriction. Note that if 
\[\phi(v_1) = (a_{11}, a_{12}, a_{13}) \in H^0(\mathcal{O}_{\mathbb{P}^2}
(n+1))^3 \]
\[\phi(v_2) = (a_{21}, a_{122}, a_{23}) \in 
H^0(\mathcal{O}_{\mathbb{P}^2}(n+1))^3 \]
then 
\[ \begin{array}{ccc}
 v_1\wedge (\phi(v_2))  & = &\begin{vmatrix}
                            a_{21} & a_{22} & a_{23}\\
                            f_1  & f_2 & f_3 \\
                            x & y & z
                            \end{vmatrix}   
\end{array} \]
\[
= a_{13}(f_2 z-f_3 y) -a_{22}(f_1 z-f_3 x) +a_{12}(f_2 z-f_3 y) 
\]
and 
\[ \begin{array}{ccc}
 v_2\wedge (\phi(v_1))  & = &\begin{vmatrix}
                            a_{11} & a_{12} & a_{13}\\
                            g_1  & g_2 & g_3 \\
                            x & y & z
                            \end{vmatrix}   
\end{array} \]
\[
= a_{23}(g_2 z-g_3 y) -a_{22}(g_1 z-g_3 x) +a_{12}(g_2 z-g_3 y). 
\]
Using these observations it is easy to see that surjectivity of $F_n$ is 
equivalent to surjectivity of $(d\Phi_n)_V.$  $\hfill{\square}$

\begin{rem}
    In general for large integer $n$ it is possibly difficult to find an 
    explicit example of  
    a two dimensional subspace $V$ of $H^0(\mathcal{O}_{\mathbb{P}^2}
    (2n+3))$ 
    obtained from  homogeneous polynomials with rational coefficients as 
    in Example 
    1), for which the the linear map 
    \[(d\Phi_n)_V: Hom(V,H^0(T_{\mathbb{P}^2}(n)/V)) \to 
    H^0(\mathcal O_{C}(2n+3))  \]
    is surjective, where $C$ is the degeneracy loci of $V.$ 
    On the other hand it is 
    easy to find example of $V$ for which $(d\Phi_n)_V$ is not 
    surjective. For example, assume  $2n+3 = 3k$ for some 
    positive integer $k$ and take $V$ to be
    the subspace of $H^0(T_{\mathbb{P}^2}(n))$ generated by the images 
    of $v_1= (z^{n+1}, x^{n+1}, 0), v_2= (0, z^{n+1}, y^{n+1}).$ 
    Then it can be seen that
    image of the multiplication map $F_n$  does not contain elements of
    the type
    $x^ky^kz^k, x^{k+1}y^kz^{k-1},$ when $k$ is large 
    enough (here $F_n$ as in Theorem(\ref{thm3})). 
    Hence we can conclude that $(d\Phi_n)_V$ is not surjective.
    \end{rem}
\section{Proof of Theorem \ref{thm3}}
 Note that $\textrm{det}(E_r(n))=\mathcal{O}_{\mathbb{P}^2}(rn+2)$. 
 By \cite{LuJu}, for each $n\geq 0$, a general curve $C$ 
 of degree $rn+2$ is the zero locus of the 
 determinant of a $(r+2)\times(r+2)$ matrix $A$ of homogeneous
 polynomials $((f_{i,j}))_{1\leq i,j\leq r+2},$ 
 with $\deg f_{i,j}=n$ if $1\leq i\leq r$, and 
 $\deg f_{i,j}=1$ if $r+1\leq i\leq r+2$. Let 
 $\underline{f_1},\underline{f_2},...,\underline{f_r}
 \in H^0(\mathbb{P}^2,\mathcal{O}_{\mathbb{P}^2}(n)^{r+2})$ 
 be given by the first $r$ rows of $A$. 

 Let $E'$ be the cokernel of $\mathcal{O}_{\mathbb{P}^2}
 (-1)^{2}\xrightarrow{s'}\mathcal{O}_{\mathbb{P}^2}^{r+2}$, where $s'$
is given by the $(r+2)\times 2$ matrix 
$((f_{r+j,i}))_{1\leq i\leq r+2, 1\leq j\leq 2}.$ 
As $C$ is general, we may assume that $s'$ is a general map 
$\mathcal{O}_{\mathbb{P}^2}(-1)^{2}\to\mathcal{O}_{\mathbb{P}^2}^{r+2}$.
In particular $s'$ is injective. As shown in the proof of 
{\cite[Theorem 5.2]{BSV}},
there is $\phi\in\textrm{Aut }\mathbb{P}^2$ such that 
$\phi^* E'\cong E.$

We have a short exact sequence of vector bundles:
\[ 0 \to O_{\mathbb{P}^2}(n-1)^2{\longrightarrow} O_{\mathbb{P}^2}
(n)^{r+2}{\to} E'(n) \to 0. \]
Let $s_i\in H^0(\mathbb{P}^2,E'(n))$ be the image of 
$\underline{f_i}$ under the last
map, for $1\leq i\leq r.$ So, $C$ is the degeneracy locus of 
$s_1,s_2,...,s_r$, that
is, the zero locus of $s_1\wedge s_2\wedge...\wedge s_r\in 
H^0(\mathbb{P}^2,\textrm{det }E'(n))$. Hence, $\phi^*C$ is 
the degeneracy locus of the sections 
$\phi^*s_1,\phi^*s_2,...,\phi^*s_r$ of $\phi^* E'(n)\cong E(n).$

\section{Proof of Theorem \ref{thm4}}
Let $E$ be a semi-abundant vector bundle on $X$ of rank $d\geq 2$. So,
there is an 
ample line bundle $H$ on $X$ such that $E$ is $H$-abundant. 
For $m>>0$, let 
$$U_m = \{  [V] \in Gr(d, H^0(E(m))) \,\, | \,\, V  \, \text{is a GPLI} 
\},$$
an open subvariety of $Gr(d, H^0(E(mH))).$ Also let 
$$G_m=\{\phi\in \textrm{Aut }X| \phi^*\textrm{det }(E(mH))\cong 
\textrm{det }(E(mH))\},$$
a closed subgroup of the group scheme $\textrm{Aut }X.$ 
Letting $g=\dim Aut^0(X),$
we see that every component of $G_m$ has dimension $\leq g.$ 

We have a natural map
\[
  U_m\times G_m
  \xrightarrow{T_m} \mathbb{P}(H^0(\mathbb{P}^2, \textrm{det }E(m)^*)),  
 \] 
 whose image contain an open subset of $\mathbb{P}(H^0(\mathbb{P}^2, 
 \textrm{det }E(mH)^*))$, by the definition of $H$-abundance. As 
 $\textrm{Aut }X$ has only countably many components and is locally 
 finite type over $\mathbb{C}$, $G_m$ also has only countably many 
 components. So, there is a component $N_m$ of $G_m$ such that the 
 restriction of $T_m$ to $U_m\times N_m$ is dominant. Comparing 
 dimensions of both sides, we get 
 \begin{equation}\label{ineq}
     h^0(\textrm{det }E(m))-1\leq d(h^0(E(m))-d)+g.
 \end{equation}

 By asymptotic Riemann-Roch (see for example {\cite[Appendix]
 {kollar2013rational}}), for $m>>0$, both sides of \eqref{ineq} are 
 polynomials in $m$ of degree $n$, RHS has leading term $d^2.H^n/n!$
 and LHS has leading term $d^n.H^n/n!$. Comparison of the leading terms
 shows $n\leq 2.$

 Now assume $n=2.$ We will use the following form of Riemann-Roch 
 theorem for vector bundles on a smooth projective surface $X$: For a 
 vector bundle $F$ of rank $e$ on $X$, we have for $m>>0$,
 $h^0(X,F(mH))$ is a polynomial in $m$ given by
 \begin{equation}\label{rr}
     h^0(X,F(mH))=\frac{eH^2}{2}m^2+mH.(2c_1(F)-eK)/2+const.
 \end{equation}

 Here $K$ is the canonical divisor class of $X$, and the term $const$ 
 is independent of $m.$

 Now applying \eqref{rr} to $F=E$, the RHS of \eqref{ineq} is given by 
 $$ \frac{d^2H^2}{2}m^2+\frac{dm}{2}H.(2c_1(E)-dK)+const.$$ Also,
 applying \eqref{rr} to $F=\textrm{det }E$, the LHS of \eqref{ineq} is 
 given by 

 $$\frac{d^2H^2}{2}m^2+\frac{dm}{2}H.(2c_1(E)-K)+const.$$
 Now \eqref{ineq} gives
 \begin{equation}\label{cute}
     H.K\leq 0.
 \end{equation}

 As $H$ is ample, this shows that for all $m>0$, $mK$ is not linearly
 equivalent to a nonzero effective divisor. In other words, if 
 $h^0(X, \mathcal{O}_X(mK_X))\neq 0$, then $mK_X\sim 0.$ So, either 
$\kappa(X)=-\infty$, or $\kappa(X)=0$ and $X$ is minimal.

Now suppose $E$ is abundant. Then \eqref{cute} holds for all ample $H$.
By continuity, we get $K.D\leq 0$ for all nef $\mathbb{Q}$-divisor $D$.
So, $-K_X\in N_1(X)_{\mathbb{R}}$ is in the dual of the nef cone of 
$N^1(X)_{\mathbb{R}}$, that is, the Mori cone $\overline{NE}(X)$. Now
regarding $-K_X$ as an element of $N^1(X)_{\mathbb{R}}$, we get $-K_X$ 
is pseudoeffective. Now by {\cite[Theorem 1.3]{LX}}, we are done.
\section{Proof of Theorem \ref{thm5}}

\underline{(1):} Let $S=\mathbb{C}[X_0,X_1,Y_0,Y_1]$, a bigraded ring 
with $\deg (X_i)=(1,0)$ and $\deg (Y_i)=(0,1)$. For integers $a$ and 
$b$, let $\mathcal{O}(a,b)$ denote the line bundle $\mathcal{O}
(a)\boxtimes\mathcal{O}(b)$ on $(\mathbb{P}^1)^2$. We identify 
$H^0((\mathbb{P}^1)^2, \mathcal{O}(a,b))$ with $S_{a,b}$, 
the $(a,b)$'th graded component of $S$. Note that $S_{a,b}$ is a 
finite-dimensional vector space, so we can regard it as a variety.

Let $H=\mathcal{O}(a,b)$ be ample, so $a,b>0.$ It suffices to show 
that for $m>>0$, the morphism 
$$\psi: S_{ma,mb}^4\to S_{2ma,2mb}$$ given by 
$\psi(f_1,f_2,f_3,f_4)=f_1f_4-f_2f_3$ is dominant. At $\underline{F}=
(F_1,F_2,F_3,F_4)\in S_{ma,mb}^4$, the tangent map of $\psi$ is given by
$$d\psi_{\underline{F}}(f_1,f_2,f_3,f_4)=F_1f_4+F_4f_1-F_2f_3-F_3f_2.$$
Here we are identifying the tangent space at any point of $S_{ma,mb}^4$
with $S_{ma,mb}^4$, as $S_{ma,mb}^4$ is a vector space.

By Remark \ref{rem1}, it suffices to find $\underline{F}\in S_{ma,mb}^4$
such that $d\psi_{\underline{F}}$ is surjective. The surjectivity of 
$d\psi_{\underline{F}}$ is equivalent to $S_{2ma,2mb}$ being contained 
in the ideal $(F_1,F_2,F_3,F_4)$ of $S$.

We claim that $\underline{F}=(X_0^{ma}Y_0^{mb}, X_0^{ma}Y_1^{mb}, 
X_1^{ma}Y_0^{mb},X_1^{ma}Y_1^{mb})$ works. To see this, it suffices to 
show that any monomial 
$f=X_0^{\alpha_0}X_1^{\alpha_1}Y_0^{\beta_0}Y_1^{\beta_1}\in 
S_{2ma,2mb}$ is divisible by some $F_k$. As $\alpha_0+\alpha_1=2ma$, 
there is $0\leq i\leq 1$ with $\alpha_i\geq ma$. Similarly, there is
$0\leq j\leq 1$ with $\beta_j\geq mb$. So, $X_i^{ma}Y_j^{mb}\mid f$. As
$X_i^{ma}Y_j^{mb}$ is some $F_k$, we are done.

\underline{(2) and (3):} Both follow from the following Lemma.
\begin{lem}\label{restriction}
    Let $X$ a smooth projective surface with torsion free Picard group,
    $E$ a homogeneous vector bundle on $X$ which is $H$-abundant for 
    some ample line bundle $H$ on $X$. If $C$ is a smooth projective
    curve in $X$, then $E|_C$ is $H|_C$-abundant.
\end{lem}
\begin{proof}
    Let $d$ be the rank of $X$. For $m>>0$, we have a surjection 
    $$H^0(X,\textrm{det }E(mH))\to H^0(C,\textrm{det }E(mH)|_C).$$ Let
    $t\in H^0(C,\textrm{det }E(mH)|_C)$ be general. Choose $\tilde{t}\in 
    H^0(X,\textrm{det }E(mH)) $ mapping to $t$ under this surjection. As
    $t$ is general, we can assume $\tilde{t}$ is also general. As $E$ is
    $H$-abundant, there is $\phi\in \textrm{Aut }X$, an isomorphism
    \begin{equation}\label{isom}
        \phi^*\textrm{det }E(mH)\cong \textrm{det }E(mH),
    \end{equation}
and sections $\tilde{s_1},\tilde{s_2},...,\tilde{s_d}$ of $E(mH)$ 
    such that $\tilde{t}=\phi^*\tilde{s_1}\wedge 
    \phi^*\tilde{s_2}\wedge...\wedge \phi^*\tilde{s_d}$. As $E$ is
    homogeneous, \eqref{isom} shows $\phi^*(dmH)\cong dmH$, hence
    $dm(\phi^*H-H)=0$ in $\textrm{Pic }X$. As $\textrm{Pic }X$ is
    torison free, we get $\phi^*H\cong H$. Hence $\phi^*E(mH)\cong 
    E(mH)$. So we can regard $\phi^*\tilde{s_i}$ as a section of $
    E(mH)$, let $s_i$ be its restriction to $C$, hence a section of
    $E(mH)|_C$. Now $t$ is a nonzero scalar multiple of $s_1\wedge 
    s_2\wedge ...\wedge s_d$, completing the proof.
\end{proof}
\begin{rem}
    By a similar proof as in $(1)$ of Theorem \ref{thm5} one can show 
    several other rank $2$ split bundles on $(\mathbb{P}^1)^2$ is
    abundant.
\end{rem}
\begin{rem}
    Any curve of genus $\leq 3$ and any hyperelliptic curve of any genus 
    satisfies the hypothesis in $(2)$ of Theorem \ref{thm5}.
\end{rem}
\printbibliography
\end{document}